\documentclass[12pt]{amsart}
\usepackage{amsmath,amsthm,amsfonts,amssymb,graphicx,amscd}
\usepackage{graphicx}
\usepackage{color}
\begin{document} 
\newcommand{\spher}[1]{{#1}^{\Delta}}
\newcommand{\fs}{\spher{f}}
\newcommand{\Fs}{\spher{F}}
\newcommand{\A}{{\mathbb A}}
\newcommand{\B}{{\mathbb B}}
\newcommand{\C}{{\mathbb C}}
\newcommand{\Cc}{{{\mathcal C}_c}}
\newcommand{\Dp}{{D_1}}
\newcommand{\Dpr}{D_{1,\R}}
\newcommand{\G}{{\mathcal G}}
\renewcommand{\H}{{\mathbb H}}
\let\paragraf\S
\renewcommand{\S}{{\mathbb S}}
\newcommand{\N}{{\mathbb N}}
\newcommand{\Q}{{\mathbb Q}}
\newcommand{\Z}{{\mathbb Z}}
\renewcommand{\P}{{\mathbb P}}
\renewcommand{\O}{{\mathcal O}}
\newcommand{\R}{{\mathbb R}}
\newcommand{\rc}{\subset}
\newcommand{\rank}{\mathop{rank}}
\newcommand{\supp}{\mathop{supp}}
\newcommand{\trace}{\mathop{tr}}
\newcommand{\tensor}{\otimes}
\newcommand{\dimc}{\mathop{dim}_{\C}}
\newcommand{\Lie}{\mathop{Lie}}
\newcommand{\Spec}{\mathop{Spec}}
\newcommand{\Auto}{\mathop{{\rm Aut}_{\mathcal O}}}
\newcommand{\End}{\mathop{{\rm End}}}
\newcommand{\alg}[1]{{\mathbf #1}}
\let\Realpart\Re
\renewcommand{\Re}{\mathop{\Realpart e}}

\let\Impart\Im
\renewcommand{\Im}{\mathop{\Impart m}}

\newcommand{\Image}{\mathop{image}}
\newtheorem*{convention}{Convention}
\newtheorem*{claim}{Claim}
\newtheorem*{Conjecture}{Conjecture}
\newtheorem*{SpecAss}{Special Assumptions}
\newtheorem*{observation}{Observation}
\newtheorem*{fact}{Fact}
\newtheorem{theorem}{Theorem}[section]
\newtheorem{proposition}[theorem]{Proposition}
\newtheorem{corollary}[theorem]{Corollary}
\newtheorem{lemma}[theorem]{Lemma}
\theoremstyle{definition}
\newtheorem{definition}[theorem]{Definition}
\newtheorem{example}[theorem]{Example}
\newtheorem{remark}[theorem]{Remark}
\numberwithin{equation}{section}
\def\labelenumi{\rm(\roman{enumi})}
\title[Runge pairs]{%
On Runge pairs and topology of axially symmetric domains
}
\author {Cinzia Bisi \& J\"org Winkelmann}
\begin{abstract}
We prove a Runge theorem for and describe the homology of
axially symmetric open subsets of $\H$.
\end{abstract}
\subjclass{30G35}%
\address{%
Cinzia Bisi \\
Department of Mathematics and Computer Sciences\\
Ferrara University\\
Via Machiavelli 30\\
44121 Ferrara \\
Italy
}
\email{bsicnz@unife.it \newline
ORCID: 0000-0002-4973-1053
}
\address{%
J\"org Winkelmann \\
Lehrstuhl Analysis II \\
Fakult\"at f\"ur Mathematik \\
Ruhr-Universit\"at Bochum\\
44780 Bochum \\
Germany
}
\email{joerg.winkelmann@rub.de\newline
  ORCID: 0000-0002-1781-5842
}
\thanks{
The two authors were partially supported by GNSAGA of INdAM.
C. Bisi was also partially supported by PRIN \textit{Variet\'a reali e complesse:
geometria, topologia e analisi armonica}. 
}
\maketitle

\section{Introduction}
Approximation theory plays a fundamental role in complex analysis, holomorphic dynamics, the theory of minimal surfaces in Euclidean spaces and in many other related fields of mathematics.
In this paper, our goal is to study quaternionic analogs of the classical
complex Runge theory, in particular analogs of the classical
topological characterization of domains in the complex plane
on which holomorphic functions may be approximated
by entire functions.
We recall that the classical theory of holomorphic approximation started in 19th century with the amazing results of Runge and Weierstrass (1885) and continued in the 20th century with the work
 of Oka and Weil, Mergelyan, Vituskin and others : here we prove the analog of Behnke and Stein theorem in the more modern quaternionic setting,
 hoping that this paper will bring a new stimulus for future developments in this important area of mathematics.

Throughout this paper the integers, real, complex and quaternionic
numbers are denoted by $\Z$, $\R$, $\C$ and $\H$ respectively.
We recall that $\H$ is a skew field, a  four-dimensional 
associative $\R$-algebra
with basis $1,I,J,K$ subject to the rules $I^2=J^2=K^2=-1$,
$IJ+JI=IK+KI=KJ+JK=0$, $IJK=-1$.

The set of imaginary units $\S=\{q\in \H:q^2=-1\}$
is a real two-dimensional sphere, because
\[
\S=\{xI+yJ+zK: x^2+y^2+z^2=1\}.
\]

Our goal is to study (slice) {\em regular functions} on domains in
$\H$ which are the
analog of holomorphic functions on $\C$.

\begin{definition}
Let $\Omega$ be an open subset of $\H$ with $\Omega\cap\R\ne\{\}$. 
A real differentiable function $f \colon \Omega \to \H $ is said to be 
(slice) regular if, $\forall \,\, I \in \S$ 
its restriction $f_{I}$ to the complex line 
$\C_I= \mathbb{R} + \mathbb{R} I$ passing through the origin and containing $1$ and $I$ is holomorphic on $\Omega \cap \C_I$.
\end{definition}

This notion was  introduced by Gentili and Struppa \cite{GS06,GS07}.

For a ball in $\H$ centered at the origin
 regularity is the same as the condition that the function
can be represented by a convergent power series
\[
f(q)=\sum_{k=0}^\infty q^k a_k.
\]

In the last decade the theory of slice regular functions has been investigated in many directions, see, as samples, the papers \cite{bisigentili}, \cite{bisigentilitrends}, \cite{bs12}, \cite{bisistoppato}, \cite{bs17}, \cite{BG}, \cite{AnB}, \cite{AB},
 \cite{BW}, \cite{BW1}.
 
In this article, we call an open subset $D\subset\C$
{\em symmetric} if it is invariant under complex conjugation.
An open subset $\Omega\subset\H$ is called {axially symmetric}
if it is invariant under all $\R$-algebra automorphisms of $\H$.
This is equivalent to the condition that for any $x,y\in\R$,
$I,J\in\S$ the condition $x+yI\in\Omega$ holds if and only if
$x+yJ\in\Omega$.

There is a one-to-one correspondence between symmetric
open subsets $D\subset\C$ and axially symmetric open
subsets $\Omega_D\subset\H$ which may described as follows.

Given an axially symmetric open subset $\Omega\subset \H$, 
we may choose an element $I\in\S$ and define 
$D\subset\C$ as
\[
D=\{x+yi:x+yI\in\Omega, x,y\in\R\}.
\]
Conversely, given a symmetric open subset $D\subset\C$, we 
define the corresponding axially symmetric subset $\Omega\subset\H$
(which we often denote as $\Omega_D$) via
\[
\Omega=\{x+yI: I\in\S,x,y\in\R, x+yi\in D\}.
\]

\newcommand{\Hc}{\H\tensor_\R \C}
Let $D$ be a symmetric open subset of $\C$. Then
a ``stem function'' on $D$ is a holomorphic function $F:D\to\Hc$
such that $F(\bar z)=\overline{F(z)}$ for all $z\in D$.
Here ``holomorphic'' is to be understood with respect to the complex
structure on $\Hc$ induced by the complex structure on the
second factor of the tensor product.

Given a symmetric open subset $D\subset\C$
with $D\cap\R\ne\{\}$ and its associated
axially symmetric open subset $\Omega_D$ we have a one-to-one
correspondence between slice regular functions on $\Omega_D$ and
``stem functions on $D$''.

Given a stem function $F:D\to\Hc$, we write $F$ as
\[
F(z)=F_1(z)\tensor 1 + F_2(z)\tensor\iota
\]
with $F_i:D\to\H$ and define
\[
f(x+yI)=F_1(x+yi)+IF_2(x+yi)\quad\quad(x,y\in\R, I\in\S)
\]

Conversely, given $f:\Omega\to\H$, we fix an element $I\in\S$ and
define
\begin{align*}
F_1(x+yi) &=\frac 12 \left( f(x+yI) +f(x-yI) \right) \\
F_2(x+yi) &= -I\frac 12  \left( f(x+yI) -f(x-yI) \right) \\
\end{align*}
It can be shown (using the ``representation formula'') that the $F_i$
are independent of the choice of $I$, see \cite{ghiloniperotti}.

For arbitrary axially symmetric domains in $\H$ (for which the intersection
with the real axis may be empty) we use the definition below.

\begin{definition}
  Let $D$ be a symmetric domain in $\C$ and let $\Omega_D$ be its
  associated axially symmetric domain in $\H$, i.e.,
  \[
  \Omega_D=\{x+yJ: x,y \in\R, J\in\S, x+yi\in D\}
  \]
  A function $f:\Omega_D\to \H$ is {\em regular} if it is
  induced by a holomorphic stem function $F:D\to\Hc$.
\end{definition}

Our main result is the following:
\begin{theorem}
Let $D\subset\Dp$ be symmetric open subsets of $\C$ and let
$\Omega_D\subset\Omega_{\Dp}$ 
be the corresponding axially symmetric open subsets
in $\H$.

Then the following are equivalent:
\begin{enumerate}
\item
$D\subset \Dp$ is a Runge pair, i.e., every holomorphic function on $D$
can be approximated by holomorphic functions on $\Dp$
(uniformly on compact sets),
\item
$\Omega_D$ is Runge in $\Omega_{\Dp}$ in the sense that every regular
function on  $\Omega_D$ can be approximated (uniformly on compact sets)
by regular functions on $\Omega_{\Dp}$.
\item
  $i_*:H_1(D)\to H_1(\Dp)$ is injective, where
  $i_*$ denotes the homology group homomorphism induced by the
  inclusion map $i:D\to\Dp$.
\item
  $i_*: H_k(\Omega_D)\to H_k(\Omega_{\Dp})$ is injective for $k\in\{1,3\}$
  where $i_*$ is the homomorphism induced by the inclusion map
  $i:\Omega_D\to\Omega_{\Dp}$.
\item
Every bounded connected component of $\C\setminus D$ intersects
$\C\setminus\Dp$.
\item
Every bounded connected component of $\H\setminus\Omega_D$ intersects
$\H\setminus\Omega_{\Dp}$.
\end{enumerate}
\end{theorem}

The equivalences $(i)\iff(iii)\iff (v)$ are classical
(see Proposition~\ref{complex-runge} below).
The implication $(vi)\Rightarrow(ii)$ has been proven before
by Colombo, Sabadini and Struppa (Theorem 4.13 of \cite{CSS}).

The equivalence $(i)\iff(ii)$ is Proposition~\ref{q-c-runge}.
The equivalence $(iii)\iff(iv)$ is Proposition~\ref{equiv-homology}.

The equivalence $(v)\iff (vi)$ is an easy consequence of the fact
that each bounded connected component $C$ of $D,$ resp.~$\Dp,$
corresponds to a bounded connected component $\Omega_C$
of $\Omega_D,$ resp.~$\Omega_{\Dp},$ via
\[
\Omega_C=\{ x+yI; x,y,\in\R, x+yi\in C, I\in\S \}.
\]

In the context of
proving our results on Runge pairs
we obtain a precise description of the homology of $\Omega_D$
in terms of the topology of $D$. (see Proposition~\ref{homology}.)

\subsection{Examples.}
\begin{example}
$\C^*$ is a symmetric domain with 
corresponding axially symmetric domain $\H^*$. 

$\H^*$ is simply-connected, but not Runge in $\H$, because
$i_*:H_3(\H^*)\simeq\Z\to H_3(\H)=\{0\}$ is not injective.
\end{example}

\begin{example}
$\C\setminus\R$ is a symmetric domain with corresponding 
axially symmetric domain 
$\Omega=\H\setminus\R$. The domain $\Omega$ is homotopic to
the $2$-sphere, thus simply-connected but not contractible. However,
$\Omega$ is Runge in $\H$: $H_1(\Omega)$ and $H_3(\Omega)$ vanish both,
hence $H_k(\Omega)\to H_k(\H)$ is injective for $k=1,3$.
Thus we have a Runge pair
although $\Z\simeq H_2(\Omega)\to H_2(\H)=\{0\}$ is not injective.
\end{example}

\begin{example}
Let $D=\{z\in\C:|z|>1\}$ and $\Dp=D\cup\{z\in\C: -1/2 < \Im(z)<1/2\}$.

Then $\Omega_D$ is Runge in $\Omega_{\Dp}$.

Evidently $\Omega_D$ is the complement of the closed unit ball in $\H$ and
therefore homotopic to the $3$-sphere.
Now $D_1\ne \C$, hence $\exists\ p\not\in\Omega_{D_1}$ and we have
inclusion maps
\[
\Omega_D\stackrel{i}{\hookrightarrow}
\Omega_{D_1}\stackrel{j}{\hookrightarrow}\H\setminus\{p\}.
\]
Since the composition map $j\circ i$ is a homotopy equivalence,
all the homology group homomorphisms $i_*$ induced by $i$ must be
injective. Hence our results imply that $D$ is Runge in $D_1$.
\end{example}

\section{Runge}

\subsection{The complex situation}

In the complex case one has the following well-known result.

\begin{proposition}\label{complex-runge}
Let $D\subset \Dp$ be open subsets of $\C$. Then the following properties
are equivalent:
\begin{enumerate}
\item
The inclusion map induces an injective group homomorphism
$H_1(D)\to H_1(\Dp)$.
\item
Every bounded connected component of $\C\setminus D$ intersects
$\C\setminus \Dp$.
\item
For every holomorphic function $f$ on $D$, every $\epsilon>0$ and
every compact subset $K\subset D$ there exists a
holomorphic function $F$ on $\Dp$ with $\sup_{p\in K}|f(p)-F(p)|<\epsilon$.
\end{enumerate}
If one (hence all) of these properties are fulfilled, then
$D\subset \Dp$ is called a {\em Runge pair}, or we say that
$D$ is Runge in $\Dp$.
\end{proposition}

See \cite{BS} and \cite{R} \paragraf13.2.1.

\subsection{Symmetric complex situation}

We recall (see \paragraf1) that a subset $D\subset\C$ is ``{\em symmetric}'' if it is
invariant under complex conjugation.

\begin{lemma}
Let $D\subset \Dp$ be symmetric open subsets of $\C$.

Then the following are equivalent:
\begin{enumerate}
\item
Every holomorphic function $f$ on $D$ can be approximated 
(locally uniformly) by
holomorphic functions on $\Dp$
(i.e., $D\subset\Dp$ is a Runge pair)
\item
Every holomorphic function $f$ on $D$ which is {\em symmetric},i.e., for which $f(z)=\overline{f(\bar z)}$ holds,
can be approximated 
(locally uniformly) by symmetric
holomorphic functions on $\Dp$.
\end{enumerate}
\end{lemma}

\begin{proof}
$(i)\implies (ii)$. Assume that $D$ is Runge in $\Dp$ and that $f:D\to \C$
is holomorphic with $f(z)=\overline{f(\bar z)}$. If $f_n$ is a 
sequence of holomorphic functions on $\Dp$ converging to $f$,
then
also
\[
g_n(z)=\frac{1}{2}\left( f_n(z) + \overline{f_n(\bar z)}\right)
\]
converges to $f$ and in addition fulfills
$g_n(z)=\overline{g_n(\bar z)}$

$(ii)\implies(i)$.
Let $f:D\to\C$ be an arbitrary holomorphic function.
We define
\begin{align*}
g(z) &= \frac 12\left( f(z)+\overline{f(\bar z)}\right)\\
h(z) &= \frac 1{2i}\left( f(z)-\overline{f(\bar z)}\right)
\end{align*}
Then $g$ and $h$ are both symmetric holomorphic functions and
$f(z)= g(z)+ih(z)$.
By assumption the functions $g$ and $h$ may be approximated by holomorphic
functions on $\Dp$. It follows that $f=g+ih$ can be approximated, too.
\end{proof}

\subsection{Passing from $D$ to $\Omega_D$}

Let a symmetric open subset $D\subset\C$ be given.
The associated axially symmetric subset $\Omega_D$ in $\H$
has been defined in $\paragraf1$ as:

\[
\Omega_D=\{x+yI:x,y\in\R, I\in\S, x+yi\in D\}
\]

(with $\S=\{q\in\H:q^2= -1\}$).

This construction may be reformulated as follows.

Define $D^+=D\cap\{z\in\C:\Im(z)\ge 0\}$, $D_\R=D\cap\R$.

Let $Z=D^+\times\S$. Then $\Omega_D\simeq Z/\!\!\sim$
where $(p,I)\sim(q,J)$ iff $p=q$
and one of the following conditions is fulfilled:
\begin{enumerate}
\item
$I=J$, or
\item
$p=q\in\R$.
\end{enumerate}
In other words, for each $p\in D_\R$, the subset $\{p\}\times\S$ of $Z$
is collapsed to one point.

\subsection{Quaternionic situation}

\begin{lemma}
  Let $f:\H\to\H$ be a slice  function induced by a stem function
  $F$.
  Then
  \[
  \frac {1}{\sqrt 2}|| F(x+yi)||
  \le \max\{ |f(x+yI)|, |f(x-yI)| \}
  \le
  \sqrt 2 || F(x+yi)||
  \]
  for every $x,y\in\R$, $I\in\S$.
\end{lemma}

\begin{proof}
  From $f(x+yI)=F_1(x+yi)+IF_2(x+yi)$ one deduces
  \begin{align*}
    &\hphantom{\implies}|f(x+yI)|\le ||F_1(x+yi)||+||F_2(x+yi)||\\
    &\implies |f(x+yI)|^2\le\left( ||F_1(x+yi)||+||F_2(x+yi)||\right)^2\\
    &\implies |f(x+yI)|^2\le ||F(x+yi)||^2 +2||F_1(x+yi)||\cdot||F_2(x+yi)||\\
    &\hphantom{\implies}\le 2 ||F(x+yi)||^2\\
    &\implies |f(x+yI)|\le \sqrt 2 ||F(x+yi||.\\
  \end{align*}
  On the other hand,
  $F_1(x+yi)=\frac 12\left( f(x+yI)+f(x-yI)\right)$
  implying
  $||F_1(x+yi)||\le  \max\{ |f(x+yI)|, |f(x-yI)| \}$.
  Similarily:
  $||F_2(x+yi)||\le  \max\{ |f(x+yI)|, |f(x-yI)| \}$.
  Combining these bounds we obtain:
  \[
  ||F(x+yi)||^2 \le 2\max \left \{||f(x+yI)||^2,||f(x-yI)||^2\right\}
  \]
  which implies the first inequality of the lemma.
\end{proof}
  
\begin{proposition}\label{q-c-runge}
Let $D\subset \Dp$ be a symmetric open subsets of $\C$ with 
corresponding axially symmetric open subsets $\Omega_D
\subset\Omega_{\Dp}$ in $\H$.

Then every regular function on $\Omega_D$ may be approximated locally
uniformly by regular functions on $\Omega_{\Dp}$
if and only if $D$ is Runge in $\Dp$.
\end{proposition} 

\begin{proof}
For any symmetric subset $C\subset D$ the corresponding subset
\[
\Omega_C=\{x+yI:\exists\ x+yi\in C, I\in\S\}
\]
of $\H$ is compact if and only if $C$ is compact.
We measure the size of a function by using the sup-norm.
From the euclidean scalar product on $\C\simeq\R^2$ and $\H\simeq\R^4$
we deduce a scalar product on $\H\tensor\C\simeq\R^8$. The norm
induced by this scalar product is denoted by $||\ ||$.
From the preceding lemma we
deduce that
\[
\frac 1{\sqrt 2}||F||_C \le ||f||_{\Omega_C} \le \sqrt 2 ||F||_C 
\]
for any compact symmetric subset $C\subset D$
(where $||F||_C=\sup_{z\in C }||F(z)||$.)
Therefore the space of slice functions on $\Omega_D$ is
isomorphic as a topological vector space to the space
of stem functions on $D$ (both spaces endowed with topology of
locally uniform convergence).
This implies the assertion.
\end{proof}

\subsection{Homology of axially symmetric domains}
In this paragraph we show that and how  the homology of an
axially symmetric domain in $\H$ is determined by that of the corresponding
symmetric open set in $\C$.

We will study the topology of this procedure aided by the
Mayer-Vietoris sequence.

We introduce some notation which we will keep throughout this
section.
\begin{convention}
Let $D$ be a symmetric open subset of $\C$
(i.e.~a domain such that $z\in D\iff \bar z\in D$), 
$D^+=\{z\in D:\Im(z)\ge 0\}$,
$D^-=\{z\in D:\Im(z)\le 0\}$,
$D_\R=D\cap\R$, $D^*=D^+\setminus\R$.
For any subset $A\subset\C$ a subset $\Omega_A$ of $\H$
is defined as
\[
\Omega_A=\{x+yI:  x,y\in\R, x+yi\in A, I\in\S\}
\]

Let the boundary of $D$ in $\C$
be denoted by $\partial D$.
Define a real positive function $h$ on $D_\R$ by
\[
h(x)= dist(x,\partial D) = \inf_{z\in\partial D}|z-x|.
\]
Using the triangle inequality, it is easy to check that $h$ is
continuous.
Furthermore, we define 
$W=\{x+yi\in\C: x\in D_\R: 0\le y < h(x)\}$,
$W^*=W\setminus D_\R$.
\end{convention}

We observe that
\begin{align*}
W&=\{x+rh(x)i: x\in D_\R, r\in [0,1[ \} \\
W^*&=\{x+rh(x)i: x\in D_\R, r\in ]0,1[ \} \\
D_\R&=\{x+rh(x)i: x\in D_\R, r=0 \}. \\
\end{align*}
Since $[0,1[$, $]0,1[$ and $\{0\}$ are all contractible, it is clear
that 
the natural inclusion maps $W^*\to W$ and  $D_\R\to W$ are
homotopy equivalences. 
The inclusion map  $D^*\to D^+$ is likewise a homotopy equivalence.

We recall the definition of $\tilde H_0$: An element $\alpha$
in $H_0(X)$ is a formal finite $\Z$-linear combination of points
$\alpha=\sum n_i\{p_i\}$ ($p_i\in X$)
and therefore admits a natural degree function
by $deg(\alpha)=\sum n_i$. The ``reduced homology group'' $\tilde
H_0$ is defined as the kernel of the degree map $H_0\to\Z$.

\begin{proposition}\label{homology}
Let $D$ be a symmetric open subset 
of $\C$.
We assume that the
corresponding axially symmetric set  $\Omega_D$ is 
connected.

Then $H_2(\Omega_D)=\{0\}$
if $D_\R\ne\{\}$ and $H_2(\Omega_D)\simeq\Z$ if
$D_\R$ is empty.

There are natural exact sequences
\begin{equation}\label{seqh3}
0 \to H_1(D^+) \to H_3(\Omega_D) \to \tilde H_0(D_\R) \to 0
\end{equation}
and
\begin{equation}\label{seqh1}
0 \to H_1(D^+) \to H_1(\Omega_D)\to 0.
\end{equation}
\end{proposition}

\begin{proof}
Observe that $\Omega_D=\Omega_{D^*}\cup\Omega_W$ and
$\Omega_{D^*}\cap\Omega_W=\Omega_{W^*}$.
This yields a Mayer-Vietoris sequence for homology:
\[
\ldots \to H_{k+1}(\Omega_D)
\to H_k(\Omega_{W^*}) \to H_k(\Omega_{D^*})\oplus H_k(\Omega_{W})
\to H_k(\Omega_D)\to\ldots
\]
We claim that there are homotopy equivalences
\[
\Omega_{W^*}\sim \S\times D_\R,\quad
\Omega_W \sim D_\R,\quad
\Omega_{D^*}\sim \S\times D^*\sim \S\times D^+.
\]
The first of these homotopy equivalences holds
because
\[
\Omega_{W^*}=\{x+yI: x\in D_\R, 0<y<h(x), I\in\S\}.
\]
We observe that
$D_\R$ is a deformation retract of $\Omega_W$.
Indeed 
\[
\Omega_{W}=\{x+yI: x\in D_\R, 0\le y<h(x), I\in\S\}
\]
may be retracted to $D_\R$ via 
\[
\Phi_s:(x+yI)\mapsto(x+syI)\ \ (0\le s \le 1).
\]
Thus $\Omega_W$ is homotopy equivalent to $D_\R$.

Finally
$\Omega_{D^*}\sim \S\times D^+$
follows from
\[
\Omega_{D^*}=\{ x+yI, x+yi\in D^*, I\in \S\}\simeq D^*\times\S
\]
and the fact that $D^+$ and $D^*$ are homotopy equivalent.

Thus our Mayer-Vietoris sequence yields this exact sequence:
{\small
\[
\ldots \to H_{k+1}(\Omega_D)
\to H_k(\S\times D_\R) \to H_k(\S\times {D^+})\oplus H_k(D_\R)
\to H_k(\Omega_D)\to\ldots
\]}

Since the homology groups of the sphere $\S$ are torsion-free,
the K\"unneth formula tells us that 
\begin{align*}
H_*(\S\times X)
 & \simeq H_*(\S)\tensor_{\Z} H_*(X) \\
& \simeq \left( H_0(\S)\tensor_{\Z} H_*(X) \right)
\oplus \left( H_2(\S)\tensor_{\Z} H_*(X) \right)\\
& \simeq
H_*(X)\oplus [\S]\cdot H_*(X) \\
\end{align*}
where $[\S]\in H_2(\S)$ is the fundamental class.

Hence

\begin{align*}
\ldots \to & H_{k+1}(\Omega_D)
\to \left( H_0(\S) \tensor H_k( D_\R) \right)
\oplus \left(  H_2(\S)\tensor H_{k-2}(D_\R) \right)\\
\to & 
\left( H_0(\S) \tensor H_k( D^+) \right)
\oplus \left( H_2(\S) \tensor H_{k-2}(D^+)\right)
\oplus H_k(D_{\R})\\
\to & H_k(\Omega_D)\to\ldots\\
\end{align*}

We know that $H_k(D_\R)=\{0\}$ for $k>0$
and $H_k(D^+)=\{0\}$ for $k>1$ for dimension reasons.

Therefore our long exact Mayer-Vietoris sequences yield the
following two exact sequences:
\begin{align}\label{MV-1}
0 &\to H_2(\S)\tensor H_1(D^+)\to 
H_3(\Omega_D)\to\\
\nonumber
&H_2(\S)\tensor H_0(D_\R) \to
H_2(\S)\tensor H_0(D^+) \to
H_2(\Omega_D)\to 0 
\end{align}
and
\begin{align}\label{MV-2}
0& \to H_0(\S)\tensor H_1(D^+) \to
H_1(\Omega_D)\to  \\
\nonumber
&H_0(D_\R)\to H_0(D^+) \oplus H_0(D_\R)
\to H_0(\Omega_D)\to 0
\end{align}

Case $(1)$.
Assume now that $D_\R$ is not empty.
Then inclusion map from $D_\R$ into $D^+$ yields a 
surjective group homomorphism $H_0(D_\R) \to H_0(D^+)$
with $\tilde H_0(D_\R)$ as kernel.
Let $\alpha$ denote the homomorphism 
$H_2(\S)\tensor H_0(D_\R) \to
H_2(\S)\tensor H_0(D^+)$ in \eqref{MV-1}.
Then the exact sequence \eqref{MV-1} can be split into two parts

\begin{equation}\label{MV-3}
0 \to H_2(\S)\tensor H_1(D^+)\to 
H_3(\Omega_D)\to \ker\alpha \to 0
\end{equation}
and
\begin{equation}\label{MV-4}
0 \to \left( H_2(\S)\tensor H_0(D_\R)\right )/\ker\alpha 
\stackrel\alpha\to 
H_2(\S)\tensor H_0(D^+) \to
H_2(\Omega_D)\to 0. 
\end{equation}

Since $\ker\alpha\simeq\tilde H_0(D_\R)$, \eqref{MV-3}
now implies \eqref{seqh3}.

Furthermore \eqref{MV-4} implies that $H_2(\Omega_D)$ is zero, because
$\alpha$ is surjective.

Case $(2)$.
Now let us discuss the case where $D_\R$ is empty. Then
$H_0(D_\R)=\{0\}$ and consequently 
from \eqref{MV-1} we  obtain two sequences

\[
0 \to H_2(\S)\tensor H_1(D_\R)\to 
H_3(\Omega_D)\to 0=H_2(\S)\tensor H_0(D_\R)
\]
and
\[
0=H_2(\S)\tensor H_0(D_\R)\to
\Z\simeq H_2(\S)\tensor H_0(D^+) \to
H_2(\Omega_D)\to 0 
\]
Using $H_2(\S)\simeq\Z\simeq H_0(\S)$ we get 
\eqref{seqh3} and $H_2(\Omega_D)=\{\Z\}$.

It remains to show \eqref{seqh1}.
For this purpose
we  return to \eqref{MV-2}.

The map $H_0(D_\R)\to H_0(D^+) \oplus H_0(D_\R)$ 
in \eqref{MV-2}
is obviously injective,
therefore (due to exactness of the sequence) the preceding map is zero
and $H_1(\Omega_D)$ is isomorphic to $H_0(\S)\tensor H_1(D^+)$.
However, $H_0(\S)\simeq\Z$ and therefore
$H_0(\S)\tensor H_1(D^*)\simeq H_1(D^*)$.
Hence $H_1(\Omega_D)\simeq H_1(D^*)$.
\end{proof}

\begin{corollary}
  Assume in addition that $D$ is a bounded domain with smooth
  boundary. Then all the homology groups are finitely generated
  and Proposition~\ref{homology} implies the following
  description of the Betti numbers $b_k=\dim_\R H_k(\ ,\Z)\tensor_{\Z}\R$:
  Let $r=b_0(D_\R)-1$ if $D_\R$ is not empty and set $r=0$ if $D_\R$
  is empty.
  Then
  \begin{align*}
    b_1(\Omega_D) &=\frac 12\left( b_1(D)-r \right)\\
    b_3(\Omega_D) &=\frac 12\left( b_1(D)+r \right)\\
    \end{align*}
  and
\[      b_2(\Omega_D) =
    \begin{cases} 1 & \text{ if $D_\R$ is empty}\\
       0 & \text{ if $D_\R$ is not empty}\\
    \end{cases}
\]
\end{corollary}
\begin{corollary}\label{cor-homology}
Let $D$ be a symmetric open subset and let $\Omega_D$ denote the
corresponding axially symmetric set (not necessarily connected).

Then $H_2(\Omega_D)\simeq\Z^k$  where $k$ denote the number of
connected components of $D^+$ which do not intersect $\R$.

Let $\hat H_0(D_\R)$ denote the kernel of the homomorphism
$i_*: H_0(D_\R)\to H_0(D^+)$.

There are natural exact sequences
\begin{equation}
0 \to H_1(D^+) \to H_3(\Omega_D) \to \hat H_0(D_\R) \to 0
\end{equation}
and
\begin{equation}
0 \to H_1(D^+) \to H_1(\Omega_D)\to 0.
\end{equation}
\end{corollary}

\begin{proof}
This is an easy consequence of Proposition~\ref{homology}, since the homology of
a disconnected space is isomorphic to the
direct sum of the homology of its connected components.
\end{proof}

\begin{corollary}
For an axially symmetric open subset $\Omega\subset\H$ 
all homology groups are torsion-free.
\end{corollary}

\begin{proof}
First observe that there is no loss in generality in assuming
that $\Omega_D$ is connected, because the homology groups of $\Omega_D$
are isomorphic to the direct sum of the homology groups of its
connected components.

For connected $\Omega_D$ the assertion
 follows from the preceding proposition, because the homology groups
of open sets in $\R$ and $\R^2$ are known to be always torsion-free
and $D_\R$, resp.~$D^*$, is an open subset in $\R$ resp.~$\R^2$.
\end{proof}

We now explain the geometric meaning of the short exact sequence
\eqref{seqh3}.
Given an element $\alpha\in H_1(D^+)$ we may represent $\alpha$
as a finite formal $\Z$-linear combination of closed curves $\gamma_j:
S^1 \to D^+$.
Each such curve $\gamma_j$ defines a map $\eta$
from $S^1\times\S$
to $\Omega_D$ via
\[
\eta(t,I)=\Re(\gamma_j(t))+I\Im(\gamma_j(t)).
\]
The fundamental class of the real three-dimensional manifold
$S^1\times\S$ then defines the corresponding element in
$H_3(\Omega_D)$.

An element $\beta\in H_0(D_\R)$ may be represented as a formal
$\Z$-linear combination of points $\sum n_i\{p_i\}$.
Assume that $\beta$ is in the kernel of the natural map
to $\Z$ which is given by $\sum n_i\{p_i\}\mapsto \sum n_i$.
Then $\beta$ is the sum of elements of the form $+1\{p_i\}-1\{q_i\}$.
Given such an element, we choose a curve $\gamma:[0,1]\to D^+$
with $\gamma(0)=p_i$, $\gamma(1)=q_i$, $\gamma(t)\in D^+\setminus
\R$ for $0<t<1$.
Then $\Omega_{\gamma([0,1])}$ is a $3$-sphere defining an element
in $H_3(\Omega_D)$. Note that this construction depends on the
choice of the curve $\gamma$. Therefore the sequence~\eqref{seqh3} has no
natural splitting.

\begin{lemma}
Let $D\subset\C$ be a symmetric open subset.

With $D^+$, $D_\R$ and $\hat H_0(D_\R)$ defined as in
Corollary~\ref{cor-homology}
there is natural exact sequence
\begin{equation}\label{seqhD}
0 \to H_1(D^+)\oplus H_1(D^-) 
\to H_1(D) \to \hat H_0(D_\R) \to 0
\end{equation}
\end{lemma}

\begin{proof}
  Let $W$ be as above in the proof of Proposition~\ref{homology}
  and define
  \begin{align*}
    & V=\{z\in\C:z\in W\text{ or }\bar z\in W\}\\
    & U^+=D^+\cup V,\quad
U^-=D^-\cup V.
  \end{align*}
Observe that we have homotopy equivalences
\[
U^+\sim D^+,\quad U^-\sim D^-,\quad (U^+\cap U^-)=V\sim D_\R
\]

We use the Mayer-Vietoris sequence associated to
$D=U^+\cup U^-$:
\[
\ldots \to H_{k+1}(D)
\to H_k(D_\R) \to H_k(D^+)\oplus H_k(D^-) \to H_k(D)\to \ldots
\]
The details (which we omit) are very much similar to the proof
of Proposition~\ref{homology}.
\end{proof}

\begin{corollary}\label{cor-inj-1}
Let $D\subset \Dp$ be symmetric open subsets in $\C$.
Assume that 
$H_1(D)\to H_1(\Dp)$ is injective. Then
$H_1(D^+)\to H_1(\Dp^+)$ is injective, too.
\end{corollary}

\begin{proof}
The inclusion map from $D$ to $\Dp$ combined with
\eqref{seqhD} yields the following commutative diagram
\[
\minCDarrowwidth10pt
\begin{CD}
0 @>>>  H_1(D^+)\oplus H_1(D^-) 
@>>>  H_1(D) @>>>  \hat H_0(D_\R) @>>>  0\\
@VVV @VVV @VVV @VVV @VVV \\
0 @>>>  H_1(\Dp^+)\oplus H_1(\Dp^-) 
@>>>  H_1(\Dp) @>>>  \hat H_0(\Dpr) @>>>  0\\
\end{CD}
\]
Now the assertion follows from the snake lemma
(see e.g.~\cite{L},III.\paragraf 9).
\end{proof}

\begin{proposition}
Let $D$ be a symmetric open subset of $\C$.

Then there is a natural exact sequence
\begin{equation}\label{seqOD}
\begin{CD}
0 @>>> H_1(D^+) @>\alpha>> H_1(D) @>\beta>> H_3(\Omega_D)@>>> 0 .
\end{CD}
\end{equation}
Here $\alpha$, $\beta$ are as follows:
Let $\tau:\C\to\C$ denote complex conjugation on $\C$ 
and let $\zeta:D\times\S\to\Omega_D$
be the map given by
\[
\zeta(x+yi,J)=x+yJ.
\]
Then $\alpha(\gamma)=\gamma-\tau_*\gamma$ and
$\beta(\gamma)=\zeta_*(\gamma\times[\S])$
where $[\S]\in H_2(\S)$ denotes the fundamental class.
\end{proposition}

\begin{proof}
There is no loss in generality in assuming that $D^+$ is connected (and
therefore $\Omega_D$, too).

We cover $D^+$ by the two open subsets $D^*$ and $W$ as in 
the proof of Proposition~\ref{homology}.

This induces corresponding coverings of $D$, $D\times\S$
and $\Omega_D$:
\begin{align*}
  & D= ( D\setminus D_\R) \ \cup V \ \text{with
    $V=\{z\in\C:z\in W\text{ or }\bar z\in W\}$}
\\
& D\times\S=((D\setminus D_\R)\times\S)\ \cup\ (V\times\S).\\
& \Omega_D= \Omega_{D^*}\ \cup\ \Omega_W.\\
\end{align*}
For each of these coverings we obtain a Mayer-Vietoris sequence
for homology.

We utilize the map $\zeta:D\times\S\to \Omega_D$
given by
\[
(x+yi;J)\mapsto x+yJ.
\]
This yields a morphism between the respective Mayer-Vietoris sequences:
\[
\footnotesize
\minCDarrowwidth10pt
\begin{CD}
 \ldots@>>> H_k((V\setminus D_\R)\times\S) @>>> H_k((D\setminus D_\R)\times\S)
\oplus H_k(V\times\S) @>>> H_k(D\times\S) @>>>\ldots \\
@.  @VVV @VVV @VVV \\
\ldots @>>> H_k(\Omega_{W^*}) @>>> H_k(\Omega_{D^*})
\oplus H_k(\Omega_W) @>>> H_k(\Omega_D) @>>> \ldots\\
\end{CD}
\]
In particular, we get
\[
\footnotesize
\minCDarrowwidth10pt
\begin{CD}
 H_3((V\setminus D_\R)\times\S) @>>> H_3((D\setminus D_\R)\times\S)
\oplus H_3(V\times\S) @>>> H_3(D\times\S) @>>> C @>>> 0 \\
@VVV  @VVV @VVV @VVV @VVV\\
 H_3(\Omega_{W^*}) @>>> H_3(\Omega_{D^*})
\oplus H_3(\Omega_W) @>>> H_3(\Omega_D) @>>> C' @>>> 0 \\
\end{CD}
\]
with 
\begin{align*}
C &= \ker[H_2((V\setminus D_\R)\times\S) \to H_2((D\setminus D_\R)\times\S)
\oplus H_2(V\times\S)]\\
\end{align*}
and
\begin{align*}
C' &= \ker[ H_2(\Omega_{W^*}) \to H_2(\Omega_{D^*})
\oplus H_2(\Omega_W)]\\
\end{align*}

Recall that $H_3(M\times\S)\simeq H_1(M)$
and $H_2(M\times\S)\simeq H_0(M)$ for any $M\subset\C$
due to K\"unneth formula and dimension reasons.
Observe also that $V\setminus D_\R$ 
is the disjoint union of two open subsets
(namely $D^+\cap (V\setminus D_\R)$ and
$D^-\cap (V\setminus D_\R)$) both of 
which are homotopic to $D_\R$.
Recall moreover that $V$ and  $D_\R$ are
homotopy equivalent.

Hence
\begin{align*}
C&\simeq \ker[H_0(V\setminus D_\R) \to H_0(D\setminus D_\R)
\oplus H_0(V)].\\
\end{align*}
and consequently
\begin{align*}
H_0(D_\R)\sim \ker[H_0(V\setminus D_\R) \to  H_0(V)].\\
\end{align*}
where the isomorphism 
may be describe as
\begin{align*}
H_0(D_\R)\ni &
\xi
=\sum_J n_j\{p_j\}
\\\mapsto &
\sum_J n_j\left(\{p_j-\epsilon\}-\{p_j+\epsilon\}\right)
\in \ker[H_0(V\setminus D_\R) \to  H_0(V)]\\
& (p_j\in D_\R)\\
\end{align*}
for a sufficiently small $\epsilon$.

Let
$\eta=\sum_J n_j\left(\{p_j-\epsilon\}-\{p_j+\epsilon\}\right)
\in \ker[H_0(V\setminus D_\R) \to  H_0(V)]$.
Then the homomorphism to $H_0(D\setminus D_\R)$ may be described
as 
$\eta\mapsto (\sum_J n_j,-\sum_J n_j)\in \Z^2\simeq
H_0(D\setminus D_\R)$.

It follows that

\begin{align*}
C&\simeq\tilde H_0(D_\R).\\
\end{align*}

Now
\begin{align*}
C' &= \ker[ H_2(\Omega_{W^*}) \to H_2(\Omega_{D^*})
\oplus H_2(\Omega_W)]\\
&\simeq\ker[ H_2(D_\R\times\S) \to H_2(D^+\times\S)
\oplus H_2(D_\R)]\\
\end{align*}
due to the homotopy equivalences 
(which were verified in the proof of Proposition~\ref{homology})
\[
\Omega_{W^*} \simeq D_\R\times\S ,\quad
\Omega_{D^*} \simeq D^+ \times \S,\quad
\Omega_W \simeq D_\R .
\]
It follows that
\begin{align*}
C' \simeq\ker[ H_0(D_\R) \to H_0(D^+)
\oplus \{0\}]\simeq\ker[ H_0(D_\R) \to H_0(D^+)]\simeq
\tilde H_0(D_\R).\\
\end{align*}

The aforementioned homotopy equivalences also imply
$H_3(\Omega_{D^*})\simeq H_1(D^+)$
and $H_3(\Omega_{W})=\{0\}$.
Combining all these facts, the above commutative diagram 
turns into the following
commutative diagram:
\[
\begin{CD}
0 @>>> H_1(D^+)\oplus H_1(D^-) @>\eta_1>> H_1(D) @>\eta_2>> \tilde H_0(D_\R) @>>> 0\\
@VVV @V\rho_1VV @V\rho_2VV @VV\rho_3=idV @VVV \\
0 @>>> H_1(D^+) @>\mu_1>> H_3(\Omega_D) @>\mu_2>> \tilde H_0(D_\R)@>>> 0 \\
\end{CD}
\]
The homomorphism $\rho_1$ is induced by the embedding
\[
D\setminus D_\R=D^+\cup D^-\longrightarrow \Omega_{D^*}
\]
and 
\[
H_3(\Omega_{D^*})\simeq H_3(D^+\times\S)\simeq H_1(D^+).
\]
Hence $\rho_1(c_1,c_2)=c_1+\tau_*c_2$ if $c_1$ is a $1$-cycle
in $D^+$ and $c_2$ a $1$-cycle in $D^-$.
In particular, $\rho_1$ is surjective with kernel
\[
\ker\rho_1=\{(c,-\tau_*c):c\in H_1(D^+)\}
\]
$\rho_2$ is defined by
\[
H_1(D)\simeq H_3(D\times\S)\ \stackrel{\zeta_*}\longrightarrow\ H_3(\Omega_D).
\]
We set $\beta=\rho_2$ and define $\alpha$ via 
$\alpha(c)=\eta_1(c,-\tau_*c)$. 
Injectivity of $\alpha$ is implied by injectivity of $\eta_1$.
To check surjectivity of $\beta$, let
$s\in H_3(\Omega_D)$. Since $\rho_3$ is an isomorphism, we find an
element $c\in H_1(D)$ with $\eta_2(c)=\mu_2(s)$. Then $s-\rho_2(c)
\in\ker\mu_2=\Image(\mu_1)$. Now $\rho_1$ is surjective. Therefore there
exists $a\in H_1(D^+)\oplus H_1(D^-)$ with 
\[
s-\rho_2(c)=\mu_1(\rho_1(a))=\rho_2(\eta_1(a))
\ \Rightarrow\ s=\rho_2\left(c+\eta_1(a)\right).
\]

Let us check that $\beta\circ\alpha=0$:
\[
\beta(\alpha(c))=\rho_2(\alpha(c))
=\rho_2(\eta_1(c,-\tau_*c))=\mu_1(\rho_1(c,-\tau_*c))=\mu_1(0)=0
\]
Finally, assume $b\in\ker\beta$. We have to show that $b$ is in the
image of $\alpha$. Now $\beta(b)=\rho_2(b)=0$ implies
\[
\mu_2(\rho_2(b))=\rho_3(\eta_2(b))=\eta_2(b)=0.
\]
Thus $b\in\ker(\eta_2)=\Image(\eta_1)$, i.e.,
there is an element $(c',c'')\in H_1(D^+)\oplus H_1(D^-)$
with $\eta_1(c',c'')=b$. Since $\mu_1$ is injective,
and $\rho_2(b)=0$, we know that 
\[
0=\rho_1(c',c'')=c'+\tau_*c''.
\]
Hence $c''=-\tau_*c'$.
It follows that $b=\alpha(c')$.
\end{proof}

\begin{corollary}\label{cor-inj-2}
Let $D\subset \Dp$ be symmetric open subsets in $\C$ 
such that $H_1(\Omega_D)\to H_1(\Omega_{\Dp})$
and $H_3(\Omega_D)\to H_3(\Omega_{\Dp})$
are both injective. Then
$H_1(D)\to H_1(\Dp)$ is injective, too.
\end{corollary}

\begin{proof}
First recall that
 $H_1(\Omega_D)\simeq H_1(D^+)$
(and $H_1(\Omega_{\Dp})\simeq H_1({\Dp}^+)$)
due to \eqref{seqh1}.

Second, we consider the following commutative diagram
induced from \eqref{seqOD} via the map $D\hookrightarrow \Dp$.
\[
\begin{CD}
0 @>>> H_1(D^+) @>>> H_1(D) @>>> H_3(\Omega_D)@>>> 0 \\
@VVV @VVV @VVV @VVV @VVV \\
0 @>>> H_1(\Dp^+) @>>> H_1(\Dp) @>>> H_3(\Omega_\Dp)@>>> 0 .
\end{CD}
\]
Now the snake lemma 
(see e.g.~\cite{L},III.\paragraf 9)
yields the statement.
\end{proof}

\begin{lemma}\label{lem-H3B}
Let $P$ be a symmetric compact connected subset of $\C$ such that
$P\cap\R$ is non-empty and connected.
Let $P'$ be a non-empty symmetric closed subset of $P$ and define
\begin{align*}
D  &= \C\setminus P, \\
\Dp  &= \C\setminus P'. \\
\end{align*}
Then
$H_3(\Omega_D)\to H_3(\Omega_{\Dp})$
is injective.
\end{lemma}
\begin{proof}
By construction we have
\[
H_1(D)\simeq\Z, \quad \tilde H_0(D_\R)\simeq\Z.
\]
Using \eqref{seqhD} it follows that $H_1(D^+)=\{0\}$.
Then we may apply \eqref{seqh3} to conclude
that $H_3(\Omega_D)\simeq\Z$.

Let $R>\max\{|z|:z\in P\}$. Regard the $3$-sphere $S$ with 
center $0$ and radius $R$ in $\H$.
Because $P$ is contained in the interior of the sphere, $S$ defines
a non-trivial homology class in $H_3(\Omega_D)$. Since $P'$ is also
non-empty and in the interior of the sphere, the homology class of
$S$ in $H_3(\Omega_{\Dp})$ is likewise non-zero.
Thus the homomorphism $i_*:H_3(\Omega_D)\to H_3(\Omega_{\Dp})$ maps
a non-trivial element of $H_3(\Omega_D)$ to a non-trivial element
of $H_3(\Omega_{\Dp})$. This implies the statement because
$H_3(\Omega_D)\simeq\Z$.
\end{proof}

\begin{proposition}\label{inj-h3}
Let $D\subset \Dp$ be symmetric open subsets of $\C$ such that the
natural homomorphism $H_1(D)\to H_1(\Dp)$ is injective.

Then
$H_3(\Omega_D)\to H_3(\Omega_{\Dp})$ is injective, too.
\end{proposition}

\begin{proof}
Assume the contrary.
Let 
\[
\alpha\in \ker\left (H_3(\Omega_D)\to H_3(\Omega_{\Dp})\right ),\ \alpha\ne 0.
\]
The injectivity of $H_1(D)\to H_1(\Dp)$ implies that 
$H_1(D^+)\to H_1(\Dp^+)$ is injective too
(Corollary~\ref{cor-inj-1}).
The inclusion map $D\to \Dp$ applied to \eqref{seqh3}
yields the following commutative diagram
\[
\begin{CD}
0 @>>> H_1(D^+) @>>> H_3(\Omega_D) @>>> \hat H_0(D_\R) @>>> 0\\
@VVV @VVV @VVV @VVV @VVV \\
0 @>>> H_1(\Dp^+) @>>> H_3(\Omega_{\Dp}) @>>> \hat H_0(\Dpr) @>>> 0.\\
\end{CD}
\]
Let $\alpha_0$ denote the image of
 $\alpha$ in $\hat H_0(D_\R)$. 
First, we claim that $\alpha_0$ can not vanish. Indeed, if
$\alpha_0=0$, then $\alpha$ is induced by an element $\beta
\in H_1(D^+)$. Evidently $\alpha\ne 0$ implies $\beta\ne 0$.
But now we obtain a contradiction, since 
$H_1(D^+)\to H_1(\Dp^+)$ and $H_1(\Dp^+)\to H_3(\Omega_{\Dp})$
are both injective, but $\alpha$ is mapped to zero in 
$H_3(\Omega_{\Dp})$. Hence $\alpha_0\ne 0$.

Second, by assumption the image of $\alpha$
in $H_3(\Omega_{\Dp})$ vanishes, implying that the image in
$\hat H_0(\Dpr)$ also vanishes.
Thus $\alpha_0$ is in the kernel of $\hat H_0(D_\R)
\to\hat H_0(\Dpr)$. 
Let $\alpha_0$ be represented by the formal $\Z$-linear combination
$\sum_{x\in I}n_x\{x\}$ where $I$ is a finite subset of $D_\R$.
Since $\alpha_0\ne 0$, but $\sum n_k=0$ (
because $\alpha$ is in the kernel of the morphism from
$H_0(D_\R)$ to $H_0(D)$),
we can find a point $q\in\R\setminus D$ such that 
\[
\sum_{p\in I; p>q}n_p\ne 0.
\]
Fix such a point $q$.
Let $B$ denote the connected component of $D^c=\C\setminus D$
containing $q$.

Fix $p_1,p_2\in I$ with $p_1<q<p_2$ and such that $I\cap]p_1,p_2[=\{\}$.

Note that $\alpha_0$ is mapped onto zero in $\hat H_0(\Dpr)$
which implies that $[p_1,p_2]$ is contained in $\Dpr$. 

Because $\alpha$ is mapped to zero in $H_0(D^+)$,
we know that $p_1$ and $p_2$ are contained in the same
connected component of $D^+$. Therefore $p_1$ and $p_2$
can be connected by a path $\gamma$ in $D^+$. This path, combined with
its image under conjugation, yields a closed curve inside $D$ which 
surrounds $q$. Therefore $B$ must be bounded, and $B\cap\R\subset
]p_1,p_2[$.

Combining the latter fact with $[p_1,p_2]\subset\Dpr$
implies that
$\R\cap ( B\setminus\Dp)=\{\}$.

Since we assumed that $H_1(D)\to H_1(\Dp)$ is injective, boundedness
of $B$ implies
$B\cap {\Dp}^c\ne\{\}$.

\begin{figure}
\includegraphics[width=\linewidth]{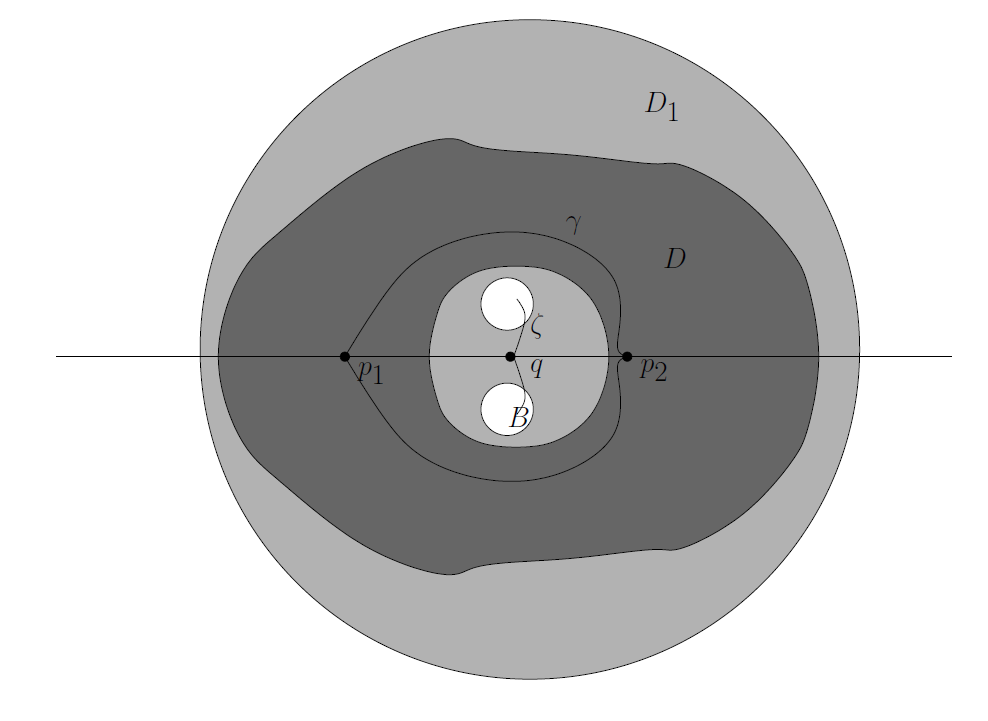}
\end{figure}

We choose a path $\zeta:[0,1]\to B$ such that
$\zeta(0)=q$, $\zeta(1)\not\in\Dp$ and $\zeta(t)\not\in\R$ for $t>0$.
Define
\[
P=\{ z\in\C: \exists\ t\in[0,1], z=\zeta(t)\text{ or }\overline{\zeta(t)}\}.
\]

Observe that $P\cap\R=\{q\}$.

Now we consider the following diagram of inclusion maps
\[
\begin{CD}
D @>>> \Dp \\
@VVV @VVV\\
\C\setminus P @>>> \C\setminus (P\cap {\Dp}^c).\\
\end{CD}
\]

From Lemma~\ref{lem-H3B} we obtain injectivity of 
\[
H_3(\Omega_{\C\setminus P})\to H_3( \Omega_{\C\setminus (P\cap {\Dp}^c)})
\]
which leads to a contradiction:  
First, by construction $\alpha_0$ is mapped to a non-zero
element of $\tilde H_0(\R\setminus P)$.
Due to \eqref{seqh3}
it follows that $\alpha$ is mapped to a non-zero
element of $H_3(\Omega_{\C\setminus P})$.
Second, its image in $H_3(\Omega_\Dp)$
is zero, which forces its image in $H_3(\Omega_{\C\setminus (P\cap \Dp^c)})$ 
to be zero, because $\Dp\subset \C\setminus (P\cap \Dp^c)$.
\end{proof}

\begin{proposition}\label{equiv-homology}
Let $D\subset \Dp$ be symmetric open subsets of $\C$
 with corresponding axially symmetric subsets $\Omega_D
\subset\Omega_{\Dp}$in $\H$.

Then $H_1(D)\to H_1(\Dp)$ is injective if and only if 
both $H_1(\Omega_D)
\to H_1(\Omega_{\Dp})$ and $H_3(\Omega_D)
\to H_3(\Omega_{\Dp})$ are injective.
\end{proposition}

\begin{proof}
First we recall that the homology of a disjoint union $X=A\cup B$
is simply the direct sum of the homology of $A$ and $B$.
For this reason there is no loss in generality in assuming that
$\Omega_D$ is connected.
 
If both $H_1(\Omega_D)
\to H_1(\Omega_{\Dp})$ and $H_3(\Omega_D)
\to H_3(\Omega_{\Dp})$ are injective, injectivity of
$H_1(D)\to H_1(\Dp)$ follows from Corollary~\ref{cor-inj-2}.

Now assume $H_1(D)\to H_1(\Dp)$ is injective.
Then $H_3(\Omega_D)
\to H_3(\Omega_{\Dp})$ is injective due to Proposition~\ref{inj-h3}.
Furthermore injectivity of $H_1(\Omega_D)
\to H_1(\Omega_{\Dp})$ follows from Corollary~\ref{cor-inj-1}
combined with \eqref{seqh1}.
\end{proof}

\section{Appendix}
\subsection{Some planar topology}
Here we show that for a pair of domains $G\subset H$ in $\C$
the group homomorphism $i_*:H_1(G)\to H_1(H)$ induced by
the inclusion map $i$ is injective if and only if every bounded
connected component of $G^c=\C\setminus G$ hits a bounded
connected component of $H^c$. This is well-known, but
we provide a new proof based on an identification of $H_1(G)$
with a certain function space, namely $\Cc(G^c,\Z)$.

\begin{proposition}
Let $G$ be an open subset of $\C$ and denote its complement by $G^c$.

Then there is a natural isomorphism $\xi$ between $H_1(G,\Z)$
and $\Cc(G^c,\Z)$
(i.e.~the space of $\Z$-valued continuous (locally constant) functions
with compact support on $G^c$).
\end{proposition}

\begin{proof}
A cycle $\gamma\in H_1(G,\Z)$  defines a function $n_\gamma$
on $\C\setminus supp(\gamma)$ by the {\em winding number}
\[
n_\gamma(z)=\int_\gamma \frac{dw}{w-z}.
\]
The winding number
$n_\gamma$ is locally constant on $\C\setminus|\gamma|$, therefore
$n_\gamma$ is continuous on $G^c$. It is compactly supported, because
$n_\gamma(z)=0$ for all $z$ with $|z|>\max\{|w|:w\in|\gamma|\}$.

Now assume that $\gamma$ is in the kernel of this map 
$\xi:\gamma\mapsto n_\gamma$. 
For each $k\in\Z$ let $Z_k$ denote the cycle defined by the open set
$\{z\in G:n_\gamma(z)=k\}$. Then the homology class of $\gamma$
in $H_1(G,\Z)$ vanishes, because $\gamma=\partial\left( \sum_k kZ_k\right)$
(here $\partial$ denotes the boundary operator in homology).
This proves injectivity of the group homomorphism  
$\xi:H_1(G,\Z)\to\Cc(G^c,\Z)$.

Conversely let $f\in\Cc(G^c,\Z)$. Since $f$ has compact support and takes
values in $\Z$, $f$ is a finite sum of functions $\pm f_i$ with
$f_i\in\Cc(G^c,\Z)$ and $f_i(z)\in\{0,1\}$ for all $z,i$.
We may therefore without loss of generality
assume that $f(G^c)=\{0,1\}$.
Let $R>\sup\{|z|:f(z)\ne 0\}$. Now we define a function $g$
on $G^c\cup\{z:|z|\ge R\}$ as follows
\[
g(z)=
\begin{cases}
f(z) & \text{if $z\in G^c$},\\
0    & \text{if $|z|\ge R$}.\\
\end{cases}
\]
We extend $g$ to a (real-valued)
smooth function $F$
defined on all of $\C$.
Sards theorem implies that $\{z:F(z)=c\}$ is a smooth submanifold
of $\C$ for almost all $c\in]0,1[$.
Each level set $\{z:F(z)=c\}$ ($0<c<1$)
is compact, because $F(z)=0$ if $|z|\ge R$.
Therefore almost every $c\in]0,1[$ defines a 
finite union of disjoint closed smooth real curves
in $\C$ which circumscribe $F=1$.
The homology class of this curve defines the element of $H_1(G,\Z)$ 
corresponding to the function $f$. 
\end{proof}

\begin{lemma}\label{3.2}
Let $A$ be a closed subset of $\C$ and let $B$ be a bounded connected
component of $A$.
Assume that $B\ne A$ and let $q\in A\setminus B$.
Then there exists a function
$f\in \Cc(A,\Z)$ which is identically $1$ on $B$ such that $f(q)= 0$.
\end{lemma}
\begin{proof}
Connected components are closed. Hence $B$ is compact. Let $R>\max\{|z|:
z\in B\}$. 

Define $C=\{z\in A:|z|=R\}$ and for each $w\in C$ choose
disjoint open subsets $U_w,V_w$ of $A$ with $A=U_w\cup V_w$, $B\subset U_w$
and $w\in V_w$. Define $f_w$ as the indicator function of $U_w$, i.e.,
\[
f_w(z)=\begin{cases} 1 & \text{if $z\in U_w$} \\
 0 & \text{if $z\in A\setminus U_w=V_w$.} \\
\end{cases}
\]
Now $C$ is a compact set covered by the open sets $V_w$ ($w\in C$).
Hence there is a finite set $S\subset C$ with
\[
C\subset\cup_{w\in S}V_w.
\]
We define
\[
g(z)=\Pi_{w\in S}f_w(z)
\]
observing that $g\equiv 1$ on $B$ and $g\equiv 0$ on $C$.

We choose a continuous function $h:A\to\{0,1\}$ such that
$h$ equals $1$ on $B$ and $h(q)=0$ (which is possible, since
$q$ lies in a connected component of $A$ different from $B$).

Now we can define the function $f$ we are looking for as
\[
f(z)=\begin{cases}
g(z)h(z) & \text{if $z\in A$ and $|z|\le R$} \\
0 & \text{if $z\in A$ and $|z|> R$.} \\
\end{cases}
\]
The function $f$ is continuous on $A$, because $g(z)=0$ for all $z\in A$
with $|z|=R$, which implies that $g(z)h(z)=0$ for $|z|=R$.
By construction its support is contained in the closed disc of radius $R$
(and therefore compact) and we have $f\equiv 1$ on $B$ and $f(q)=0$.
\end{proof}

\begin{proposition}\label{planar-top}
Let $G\subset H\subset\C$ be open subsets.
Then the following properties are equivalent:
\begin{enumerate}
\item
$H^c=\C\setminus H$ intersects each bounded connected component of $G^c$.
\item
The restriction map from $\Cc(G^c,\Z)$ to $\Cc(H^c,\Z)$ is injective.
\item 
$H_1(G,\Z)\to H_1(H,\Z)$ is injective.
\end{enumerate}
\end{proposition}
\begin{proof}
The equivalence of properties $(ii)$ and $(iii)$ has been shown above.

We prove the equivalence of $(i)$ and $(ii)$.
Let $B$ be a bounded connected component of $G^c$ with $B\subset H$.
Let $f\in\Cc(G^c,\Z)$ be a function which equals $1$  on $B$
and assumes only $0$ and $1$ as values.
(Such a function exists due to Lemma~\ref{3.2}).
Let 
\[
K=\supp(f)=\overline{\{z:f(z)\ne 0\}}
\]
be its support and define $C=K\setminus H$.
For every $x\in C$ we choose a function $g_x\in
\Cc(G^c,\Z )$ with $g_x(x)=0$ and $g_x\equiv 1$ on $B$.
(This is possible by Lemma~\ref{3.2}, since $B$ is compact).
Due to compactness of $C$ we may
choose a finite subset $S$ of $C$ such that 
\[
C\subset\cup_{x\in S}
\{z\in G^c:g_x(z)=0\}.
\]
Define
\[
g(z)= f(z)\cdot \Pi_{x\in S} g_x(z).
\]
Then $g$ equals one on $B$ and vanishes identically on $C$.
Since $\supp(g)\subset\supp(f)\subset K$, $C=K\setminus H$
and $g|_C\equiv 0$, it is clear that $g$ vanishes identically
on $H^c$.
Thus we have found a non-zero function 
$g\in\Cc(G^c,\{0,1\})$ whose restriction to $H^c$ is zero.
Therefore the existence of a bounded connected component $B$
of $G^c$ with $B\subset H$ implies that
the restriction homomorphism $\Cc(G^c,\Z)\to\Cc(H^c,\Z)$
is not injective.

To prove the opposite direction, let us assume that $B\cap H^c\ne\{\}$
for every bounded connected component $B$ of $G^c$.
Let $f\in\Cc(G^c,\Z)$. Since $f$ is locally constant and has compact 
support, it must vanish identically on every unbounded connected
component of $G^c$.
Thus, if $f\not\equiv 0$, there must be a bounded connected component
$B$ of $G^c$ on which $f$ is not zero. Since by assumption $B\cap H^c$
is not empty, it follows that the restriction of $f$ to $H^c$
is not everywhere zero. This proves injectivity.
\end{proof}

\end{document}